\documentclass[a4paper, 10pt]{amsart}

\usepackage{amsmath, amssymb, amsthm, amsfonts}
\usepackage[utf8]{inputenc}
\usepackage[T1]{fontenc}
\usepackage{lmodern}
\usepackage[UKenglish]{babel}
\usepackage[inline, shortlabels]{enumitem}
\usepackage{comment}
\usepackage{framed}
\usepackage[a4paper]{geometry} %a4 paper, proper
\usepackage[textsize = scriptsize, bordercolor = black, linecolor = blue!30, backgroundcolor = blue!10, colorinlistoftodos]{todonotes} %todo
\newcounter{todocounter}

\usepackage{csquotes}
\usepackage{bookmark}
\hypersetup{colorlinks = true, linkcolor = , citecolor = }

\newcommand{\ACF}{\mathrm{ACF}}
\newcommand{\bF}{\mathbb{F}}

\newcommand{\bK}{\mathbb{K}}
\newcommand{\bL}{\mathbb{L}}

\newcommand{\End}{\operatorname{End}}
\newcommand{\Endog}{\operatorname{Endog}}
\newcommand{\DefEnd}{\operatorname{DefEnd}}
\newcommand{\DefEndog}{\operatorname{DefEndog}}
\newcommand{\im}{\operatorname{im}}

\newcommand{\op}{\mathrm{op}}

\usepackage{scalerel}

\newcommand{\kat}{\operatorname{kat}}
\newcommand{\Kat}{\operatorname{Kat}}
\newcommand{\Frac}{\operatorname{Frac}}

\newcommand{\CC}{C^\sharp}

\theoremstyle{plain}
\newtheorem*{setting*}{Setting}
\newtheorem*{theorema}{Theorem A}
\newtheorem*{theoremb}{Theorem B}
\newtheorem{theorem}{Theorem}[section]
\newtheorem{corollary}[theorem]{Corollary}
\newtheorem{lemma}[theorem]{Lemma}
\newtheorem{proposition}[theorem]{Proposition}
\theoremstyle{definition}
\newtheorem*{definition*}{Definition}
\newtheorem{definition}[theorem]{Definition}

\newtheorem*{remark*}{Remark}
\newtheorem{remark}[theorem]{Remark}
\newtheorem{remarks}[theorem]{Remarks}
\author{Adrien Deloro and Frank Wagner}
\title{Endogenies and Linearisation}
\begin{document}
\begin{abstract}
We show that the action of two infinite commuting invariant rings of endomorphisms of a finite-dimensional virtually connected irreducible bi-module linearizes into a vector space over a definable field. The same holds if the action is merely by strongly commuting endogenies, modulo some finite katakernel.
\end{abstract}

\keywords{linearisation, irreducible bi-module, finite-dimensional theory, endogeny}
\subjclass[2000]{03C60, 03C45, 13L05}

\maketitle

\begin{flushright}
\begin{minipage}{.55\textwidth}%{.49\textwidth}
\small
\emph{--- Para la otra vez que lo mate --- replic\'o Scharlach --- le prometo ese laberinto, que consta de una sola l\'inea recta y que es invisible, incesante.}

\hfill
La muerte y la br\'ujula
\end{minipage}
\end{flushright}

\section*{Introduction}

The present work studies---under the model-theoretic hypothesis of finite-dimensionality---irreducible bi-modules, i.e.\ abelian groups together with two commuting subrings of endomorphisms, and which are irreducible for the bi-action. This enhances the result in \cite{DZilber}, which itself stems from our initial investigations.

In fact, we are working with a more general notion than endomorphism, namely \emph{endogenies}, which are endomorphisms {\em blurred} by a finite subgroup. We work in finite-dimensional theories in the sense of \cite{wa20} (see Definition \ref{d:dimensionality}), and contrary to customary practice, in the category of {\em virtually} connected groups (i.e., connected-by-finite). We first show that in a finite-dimensional theory, an irreducible virtually connected $(\Gamma,\Delta)$-bi-module $A$ admits a quotient by a finite subgroup $F$ such that $\Gamma$ and $\Delta$ act on $A/F$ as actual endomorphisms rather than endogenies (Theorem~A), and then characterise this action in Theorem~B.

\begin{remark*}
Irreducibility is with respect to the bi-action, and in the category of virtually connected bi-modules. In line with common model-theoretic teminology, we shall call this $(\Gamma,\Delta)$-minimality.
In case $\Gamma$ and $\Delta$ consist of \emph{endogenies} as discussed in Section~\ref{S:endogenies},
\emph{invariance} splits into a weak and a full form; the former is more important, and minimality is with respect to it (see Section~\ref{s:invariance}). Commutation is in the sharp sense of Section~\ref{s:sharpcommutation}\end{remark*}

We shall show:
\begin{theorema}
In a finite-dimensional theory, let $A$ be a definable, virtually connected, abelian group. Let $\Gamma$ and $\Delta$ be two invariant prerings of definable endogenies of $A$ such that $A$ is $(\Gamma,\Delta)$-minimal. Suppose $\Gamma$ and $\Delta$ commute sharply, both are essentially infinite, and one of them is essentially unbounded.
Then there is a finite fully $\Gamma$- and $\Delta$-invariant subgroup $F$ such that the induced action of $\Gamma$ and $\Delta$ on $A/F$ is by endomorphisms.\end{theorema}
In fact, $F$ is given explicitly.

\begin{theoremb}
In a finite-dimensional theory, let $A$ be a definable, connected, abelian group. Let $\Gamma$ and $\Delta$ be two invariant rings of definable endomorphisms of $A$ such that $A$ is $(\Gamma, \Delta)$-minimal. Suppose that $\Gamma$ and $\Delta$ commute, both are infinite and one of them is unbounded. Then there is a definable field $\bK$ such that $A$ is definably a $\bK$-vector space of finite linear dimension, and both $\Gamma$ and $\Delta$ act $\bK$-linearly.
\end{theoremb}

Endogenies are discussed at length in Section~\ref{S:endogenies}. More specifically, prerings are explained in Section~\ref{s:prering}; essential sizes are discussed in Section~\ref{s:equivalence}; sharp commutation is in Section~\ref{s:sharpcommutation}; then the global katakernel is defined in Section~\ref{s:radical}.
Theorem A is proven in Section~\ref{S:radicalisation} and Theorem B in Section~\ref{S:linearisation}.

To our knowledge, this is the first time endogenies are used for linearisation. Linearisation using \emph{endomorphisms} is now common and goes back to Zilber (see \cite[\S~1.2]{DZilber} for an attempted history; we could however not locate Zilber's famous `\textsc{viniti} preprint').

\subsection*{Acknowledgements.}
The first author has a variety of institutions to thank: Stellenbosch University, the Hausdorff Research Institute for Mathematics, and Université Claude Bernard Lyon 1.
The second author would like to acknowledge the support of the Science Committee of the Ministry of Science and Higher Education of the Republic of Kazakhstan (Grant No. AP19677451).

\section{Endogenies}\label{S:endogenies}
This section describes a variation on endomorphisms called \emph{endogenies}, which are like morphisms but `blurred' by finite subgroups.
Some of their properties can disconcert at first and one may wish to spend a little time computing and checking statements in order to become familiar with them. We shall add some definability and connectedness in Sections~\ref{s:dimensionality} and~\ref{s:connectedness}; the beginning is purely algebraic.

Endogenies abound in the theory of elliptic curves, with which we assume no familiarity.
A simpler example arises as follows. Let $\bK \models \ACF_p$ and $A = (\bK, + )$. The Artin-Schreier endomorphism of $A$, viz., $f(x) = x^p - x$, is onto with a finite kernel. Its inverse relation $f^{-1}$ has domain $A$ and is additive (viz., is a subgroup of $A\times A$), but behaves like a $p$-valued polymorphism. For instance it `takes' $0$ to the subgroup $\bF_p$.
%, which is the katakernel of $f^{-1}$. 
To gain intuition through practice, one should compose $f^{-1}$ with itself, and then with $f$ on both sides.

To our knowledge endogenies made their appearance in model theory (under the name of {\em quasi-endomorphism}) in the beautiful paper \cite{BCM79} which proved that $\aleph_0$-categorical $\aleph_0$-stable connected groups are abelian. The arise quite naturally: If an abelian group $A$ decomposes as almost linearly independent sum of definable minimal groups, then any definable group gives rise to endogenies between the components, and is itself characterized by these endogenies. Ever since this has been a major ingredient in the study of $\aleph_0$-categorical tame groups \cite{W93, EW00,DW20}.

\subsection{The definition}

The general definition could be formulated with non-commutative groups, but we shall quickly require abelianity. We therefore assume it from the outset.

Let $A_1, A_2$ be abelian groups.

\begin{definition}
An \emph{endogeny} from $A_1$ to $A_2$ is a 
subgroup $\gamma \leq A_1 \times A_2$ with the following properties:
\begin{itemize}
\item
$\gamma$ is global, viz., $\pi_1(\gamma) = A_1$ (where $\pi_i\colon A_1 \times A_2 \to A_i$ is the canonical projection);
\item
the \emph{katakernel} $\kat \gamma = \{a\in A_2:(0,a)\in\gamma\}$
%\pi_2(\pi_1^{-1}(0)) \leq A_2$ 
is finite.
\end{itemize}
If $A_2 = A_1 = A$, we simply call $\gamma$ an endogeny {\em of} $A$.
\end{definition}

\begin{remarks}\leavevmode
\begin{itemize}
\item
In Section~\ref{s:equivalence} we shall introduce a natural equivalence relation on endogenies.
\item
(The graph of) a morphism from $A_1$ to $A_2$ is also an endogeny; the converse does not hold.
\item
An endogeny from $A_1$ to $A_2$ may also be viewed:
\begin{itemize}
\item
as an additive relation $\gamma$ on $A_1 \times A_2$ with $\pi_1\gamma = A_1$ and finite fibre over $0\in A_1$;
\item
as a multi-but-finite-valued morphism $A_1 \to A_2$;
\item
or as a (classical) morphism $g\colon A_1 \to A_2/\kat \gamma$.
%
%There is a one-one correspondence between endogenies from $A_1$ to $A_2$, and pairs $(F, g)$ where $F \leq A_2$ is a finite subgroup and $g\colon A_1 \to A_2/F$ is a morphism.
\end{itemize}
\end{itemize}
\end{remarks}

Be careful that the image of an element $a_1 \in A_1$ under $\gamma$ is a \emph{set}, namely a coset of $\kat \gamma$. To avoid confusion we write $\gamma[a_1]$, with square brackets. Hence $\kat \gamma = \gamma[0]$ and $\im \gamma = \gamma[A]$.
Before going any further, one should check under all interpretations above that $\gamma^{-1}[\kat \gamma] = \ker \gamma$.

\subsection{The prering of endogenies}\label{s:prering}

\begin{definition}
Let $A$ be an abelian group, and $\Endog(A)$ the set of endogenies of $A$. We equip $\Endog(A)$ with the following two operations:
\begin{itemize}
\item
pointwise addition/subtraction, viz.:
\[(\gamma_1 + \gamma_2) : = \{(a, b_1 + b_2): b_1 \in \gamma_1[a] \wedge b_2 \in \gamma_2[a]\};\]
\item
product as composition of relations, viz.:
\[(\gamma_1 \circ \gamma_2) : = \{(a, c): \exists b\, (b \in \gamma_2[a]) \wedge (c \in \gamma_1[b])\}.\]
\end{itemize}
\end{definition}

One should quickly check that $\Endog(A)$ is closed under these operations. We omit $\circ$ in practice. One sees easily that $\kat(\gamma_1 + \gamma_2) = \kat \gamma_1 + \kat \gamma_2$ and $\kat (\gamma_1 \gamma_2) = \gamma_1[\kat \gamma_2]$.

\begin{remark}
Functional interpretation:
\begin{itemize}
\item
Each $\gamma_i$ corresponds to a morphism $g_i\colon A \to A/\kat \gamma_i$, which composes to $h_i\colon A\to A/(\kat \gamma_1 + \kat \gamma_2)$. The sum $\gamma_1 + \gamma_2$ corresponds to $h_1+h_2:A\to A/(\kat \gamma_1 + \kat \gamma_2)$.
\item
Also, $\gamma_1$ induces $g'_1\colon A/\kat \gamma_2 \to A/\gamma_1[\kat \gamma_2]$, and $\gamma_1 \gamma_2$ corresponds to $g'_1 g_2:A\to A/\gamma_1[\kat\gamma_2]$.
\end{itemize}
\end{remark}
In the sequel, we shall avoid these alternative points of view.

\begin{lemma}[Prering Lemma]\label{l:almostaring}
$(\Endog(A); + , \circ)$ is a \emph{prering}, in the following sense:
\begin{itemize}
\item
$(\Endog(A); + )$ is associative and commutative; the zero morphism is neutral;
\item
$(\Endog(A); \circ)$ is associative; the identity morphism is neutral;
\item
right-distributivity holds, viz., $\gamma(\delta_1 + \delta_2) = \gamma \delta_1 + \gamma \delta_2$;
\item
left-distributivity almost holds. More precisely, one has $(\gamma_1 + \gamma_2)\delta \leq \gamma_1 \delta + \gamma_2 \delta$ but only $\gamma_1 \delta + \gamma_2\delta \leq (\gamma_1 + \gamma_2) \delta + \gamma_2[\kat \delta]=
(\gamma_1 + \gamma_2) \delta + \gamma_1[\kat \delta]$.
\end{itemize}
\end{lemma}
\begin{proof}
The first claims are clear and we focus on (partial) distributivity.
\begin{itemize}
\item
Let $b \in \gamma (\delta_1 + \delta_2)[a]$. By definition, there are $x_1, x_2$ such that $x_i \in \delta_i[a]$ and $b \in \gamma[x_1 + x_2]$. Since endogenies are global, we take $b_i \in \gamma[x_i]$. Then $b_i \in \gamma \delta_i[a]$ and $b_1 + b_2 \in (\gamma \delta_1 + \gamma \delta_2)[a]$.

By definition, $b_1 + b_2 \in \gamma[x_1 + x_2]$; so $b - b_1 - b_2 \in \kat \gamma \leq \kat (\gamma \delta_1) \leq \kat (\gamma \delta_1 + \gamma \delta_2)$. Hence $b \in (\gamma \delta_1 + \gamma\delta_2)[a]$.
\item
Let $b \in (\gamma \delta_1 + \gamma \delta_2)[a]$. By definition, there are $b_i \in \gamma \delta_i[a]$ with $b = b_1 + b_2$; this is witnessed by $x_i \in \delta_i[a]$ and $b_i \in \gamma[x_i]$. Then $x_1 + x_2 \in (\delta_1 + \delta_2)[a]$ and $b = b_1 + b_2 \in \gamma[x_1 + x_2] \leq \gamma(\delta_1 + \delta_2)[a]$, as claimed.
\item
Let $b \in (\gamma_1 + \gamma_2)\delta[a]$. Then there are $b_1, b_2, x$ with $x \in \delta[a]$, $b_i \in \gamma_i[x]$, and $b = b_1 + b_2 \in \gamma_1\delta[a] + \gamma_2 \delta[a] = (\gamma_1 \delta + \gamma_2 \delta)[a]$.
\item
Finally let $b \in (\gamma_1 \delta + \gamma_2 \delta)[a]$. By definition there are $b_i$ and $x_i$ with $b = b_1 + b_2$, $x_i \in \delta[a]$ and $b_i \in \gamma_i[x_i]$. Since endogenies are global, there is $y \in \gamma_2[x_1]$. Then $b_1 + y \in (\gamma_1 + \gamma_2)[x_1] \leq (\gamma_1 + \gamma_2)\delta[a]$.

We now estimate $b - b_1 - y$. By construction, $x_1 - x_2 \in \kat \delta$. Moreover $y \in \gamma_2[x_1]$ and $b_2 \in \gamma_2[x_2]$ so $b - b_1 - y = b_2 - y \in \gamma_2[x_2 - x_1] \leq \kat \gamma_2[\delta]$, as claimed.
\qedhere
\end{itemize}
\end{proof}

We shall accordingly use the phrase \emph{sub-prering} for a subset of $\Endog(A)$ which is closed under $ +$, $-$ and $ \cdot$.

\subsection{Equivalence of endogenies}\label{s:equivalence}

Of course $\gamma - \gamma = z_{\kat\gamma}$, the endogeny taking all $A$ to $\kat \gamma$. As a relation, this is simply $A \times \kat \gamma$. This begs for an equivalence relation.

\begin{definition}
Two endogenies $\gamma_1, \gamma_2$ are \emph{equivalent} if there is a finite $F \leq A$ such that for all $a \in A$, one has $\gamma_1[a] + F = \gamma_2[a] + F$. We shall denote equivalence by $\gamma_1\sim\gamma_2$.
\end{definition}
This clearly is an equivalence relation.

\begin{remark}
Alternative ways to rephrase this condition are:
\begin{itemize}
\item
there is a finite subgroup $F\le A$ such that
$\gamma_1+z_F=\gamma_2+z_F$;
\item
there is a finite subgroup $F\le A$ such that
$\gamma_1-\gamma_2\le A\times F$ as subgroups of $A^2$;
\item
there is a finite subgroup $F\le A$ such that
$\kat\gamma_1+\kat\gamma_2\le F$ and $\gamma_1$ and $\gamma_2$ induce the same morphism $A \to A/F$. (Note that $F=\kat\gamma_1+\kat\gamma_2$ is minimal with this property.)
\end{itemize}
\end{remark}

\begin{remarks}\leavevmode
\begin{itemize}
\item
There even is a preordering $\gamma_1 \preceq \gamma_2$ whose associated equivalence is the above. Morphisms are minimal endogenies in this preordering, but there are other minimal endogenies.
\item
Not every endogeny is equivalent to a morphism (e.g., the inverse of the Artin-Schreier morphism).
\end{itemize}
\end{remarks}

\begin{corollary}[to Lemma~\ref{l:almostaring}]
Equivalence of endogenies %$\sim$ 
is preserved under $ + $ and $\circ$. Moreover $\Endog(A)/{\sim}$ is a ring.
\end{corollary}
\begin{proof}
The first is tedious and easy. Then by Lemma~\ref{l:almostaring}, distributivity holds modulo equivalence.
\end{proof}

In practice the prering $\Endog(A)$ and its quotient ring $\Endog(A)/\!\sim$ compete with each other. Model theory tends to favour the former as closer to definability issues; algebraic arguments however take place in the latter.

\begin{definition}
A set $X \subset \Endog(A)$ is \emph{essentially infinite} if its image is infinite in $\Endog(A)/{\sim}$.
\end{definition}

Later, when dealing with model-theoretic unboundedness, we shall use the phrase `essentially unbounded'.

\subsection{Sharp commutation of endogenies}\label{s:sharpcommutation}

The usual notion of commutation (as relations) turns out to be too restrictive for our purposes.
On the other hand, commutation modulo equivalence is too weak.
We strike the balance as follows.

\begin{definition}
Two endogenies $\gamma$ and $\delta$ \emph{commute sharply} if
\[\im (\gamma \delta - \delta \gamma) \leq \kat \gamma + \kat \delta\]
(equality then follows).
We write $\CC(\gamma)$ for the set of endogenies commuting sharply with $\gamma$.
\end{definition}

\begin{remark}
The condition is more restrictive than finiteness of $\im (\gamma\delta - \delta\gamma)$. In particular, not every endogeny commutes with itself; for instance the inverse endogeny of $(x\mapsto x^p - x)$ does not. (A morphism of course does.)
\end{remark}

\begin{lemma}[Sharp Commutation Lemma]\label{l:cokerinvariant}\leavevmode
\begin{enumerate}%[label = (\roman*)]
\item\label{i:cokerinvariance}
If $\gamma$ and $\delta$ commute sharply, then $\delta[\kat \gamma] \leq \kat \gamma + \kat \delta$.
\item\label{i:commutationprering}
$\CC(\gamma) \subseteq \Endog(A)$ is a sub-prering of $\Endog(A)$, viz., is closed under $ + $, $-$ and $\circ$.
\end{enumerate}
\end{lemma}
\begin{proof}\leavevmode
\begin{enumerate}%[label = (\roman*)]
\item
Let $a \in \kat \gamma=\gamma[0]$, and $b \in \delta[a]$. Then $b \in \delta\gamma[0]$, and $0 \in \gamma\delta[0]$. By sharp commutation,
\[b = b - 0 \in \im (\delta \gamma - \gamma \delta) \leq \kat \gamma + \kat \delta.\] Hence $\delta[\kat \gamma] \leq \kat \gamma + \kat \delta$.
\item
%The proof takes places in $\Endog(A)$, not in the quotient modulo equivalence of endogenies.
%??? Of course. This is also true for item 1.
Suppose $\delta_1, \delta_2 \in \CC(\gamma)$.
\begin{itemize}
\item
By Lemma~\ref{l:almostaring} and sharp commutation:
\begin{align*}
\im (\gamma (\delta_1 + \delta_2) - (\delta_1 + \delta_2) \gamma) & \leq \im (\gamma \delta_1 + \gamma \delta_2 - \delta_1 \gamma - \delta_2 \gamma)\\
& \leq \im (\gamma \delta_1 - \delta_1 \gamma) + \im (\gamma \delta_2 - \delta_2 \gamma)\\
& \leq \kat \gamma + \kat \delta_1 + \kat \delta_2\\
& \leq \kat \gamma + \kat(\delta_1 + \delta_2).
\end{align*}
\item
Note that  $\delta_1[\kat \gamma] \leq \kat \gamma + \kat \delta_1$ and $\gamma[\kat(\delta_1\delta_2)]\le\kat\gamma+\kat(\delta_1\delta_2)$ by 1. Moreover $\kat \delta_1 \leq \delta_1[\kat \delta_2] = \kat(\delta_1 \delta_2)$.
Thus:
\begin{align*}
\im (\gamma \delta_1 \delta_2 - \delta_1 \delta_2 \gamma) & \leq \im (\gamma \delta_1 \delta_2 - \delta_1 \gamma \delta_2 + \delta_1 \gamma \delta_2 - \delta_1 \delta_2 \gamma)\\
& \leq \im (\gamma \delta_1 - \delta_1 \gamma)+\gamma\delta_1[\kat\delta_2] + \delta_1[\im(\gamma \delta_2 - \delta_2 \gamma)]\\
& \leq \kat \gamma + \kat \delta_1 + \gamma[\kat(\delta_1\delta_2)]+\delta_1[\kat \gamma + \kat \delta_2]\\
& \leq \kat \gamma + \kat \delta_1 + \kat(\delta_1\delta_2)+ \delta_1[\kat\gamma] + \delta_1[\kat \delta_2]\\
&\leq \kat\gamma+\kat(\delta_1\delta_2).\qedhere
\end{align*}
\end{itemize}
\end{enumerate}
\end{proof}

\subsection{Invariance}\label{s:invariance}

As one expects, invariance downgrades to a weak form, allowing for a katakernel-error term. (Irreducibility consequently upgrades to a stronger form; see the definition of minimality in Section~\ref{s:connectedness}.)

\begin{definition}
Let $A$ be an abelian group, $B\le A$ a subgroup, and $\gamma$ an endogeny of $A$.\begin{itemize}
\item $B$ is
\emph{weakly $\gamma$-invariant} if $\gamma[B] \leq B + \kat \gamma$;
\item
$B$ is 
\emph{fully $\gamma$-invariant} if $\gamma[B] \leq B$.\end{itemize}
\end{definition}

\begin{remark} The sum of two (weakly/fully) $\gamma$-invariant subgroups is still (weakly/fully) $\gamma$-invariant, and the intersection of two fully $\gamma$-invariant subgroups is fully $\gamma$-invariant.
However, \emph{the intersection of two weakly $\gamma$-invariant subgroups need not be weakly $\gamma$-invariant}.

In general, weak invariance is the more useful notion of the two notions, which will be used in most instances. (If $\gamma$ is a morphism, then full $\gamma$-invariance and weak $\gamma$-invariance are the same.)
\end{remark}

\begin{lemma}[Invariance Lemma]\label{l:invariance}
Suppose $\gamma$ and $\delta$ are two sharply commuting endogenies of $A$.
Then:
\begin{enumerate}%[label = (\roman*)]
\item\label{i:cokerinvariance:reformulation}
$\kat \gamma$ is weakly $\delta$-invariant;
\item\label{i:gammaBdeltainvariant}
if $B \leq A$ is weakly $\gamma$-invariant, then so is $\delta[B]$.
\end{enumerate}
\end{lemma}
\begin{proof}\leavevmode
\begin{enumerate}%[label = (\roman*)]
\item
This is a reformulation of Lemma~\ref{l:cokerinvariant}.\ref{i:cokerinvariance}.
\item
By sharp commutation, $\gamma\delta[B] \leq \delta\gamma[B] + \kat \gamma + \kat \delta$. By assumption $\gamma[B] \leq B + \kat \gamma$. Since $\delta[\kat \gamma] \leq \kat \gamma + \kat \delta$ by part~\ref{i:cokerinvariance:reformulation}.\ and  $\kat\delta \leq \delta[B]$, we are done.
\qedhere
\end{enumerate}
\end{proof}

\begin{remarks}\label{r:inverse}\leavevmode
\begin{itemize}
\item
Nothing like~\ref{i:gammaBdeltainvariant}. holds with inverse images. This has annoying consequences:
\begin{itemize}
\item
$\ker \gamma$ need not be weakly $\delta$-invariant (however see Remarks \ref{r:connectedness});
\item
even if $\delta$ is invertible, $\delta^{-1}$ need not commute sharply with $\gamma$.
\end{itemize}
\item
However, \emph{if $\gamma$ is a morphism}, then $\ker \delta$ is fully $\gamma$-invariant.

Indeed if $a \in \ker \delta$, then
$$\delta[\gamma(a)]\le\gamma\delta[a]+\kat\gamma+\kat\delta\le\gamma[\kat\delta]+\kat\delta=\kat\delta,$$
so $\gamma(a)\in\delta^{-1}[\kat \delta]=\ker\delta$.
\end{itemize}
\end{remarks}

\begin{lemma}[Restriction Lemma]\label{l:restriction}
If $B \leq A$ is weakly $\gamma$-invariant, then the restriction-corestriction %$\check\gamma = \gamma_{|B}^{|B} 
$\rho_B(\gamma)= \gamma \cap B^2 \leq B^2$ is an %global
endogeny of $B$.
\end{lemma}

Of course one then has $\kat \rho_B(\gamma) \leq \kat \gamma\cap B$. If $B$ is clear from the context, we also write $\rho_B(\gamma)=\check\gamma$.

\begin{proof}
Since $\gamma[B] \leq B + \kat \gamma$, for every $b \in B$ there is $b' \in B\cap \gamma[b]$; so $\check\gamma$ is global on $B$. It clearly is a subgroup of $B^2$, and the katakernel is finite.
\end{proof}

The map $\gamma \mapsto \check\gamma$ need not be additive nor multiplicative, except modulo equivalence.

\begin{remark}\label{r:subcommuting} If $\gamma$ and $\delta$ are two sharply commuting endogenies of $A$ and $B\le A$ is weakly $\gamma$- and $\delta$-invariant, there is no reason why $\rho_B(\gamma)$ and $\rho_B(\delta)$ should commute sharply. They do, however, commute modulo equivalence.
\end{remark}

\subsection{The global katakernel}\label{s:radical}

\begin{definition}
Let $\Gamma$ and $\Delta$ be two sharply commuting prerings of endogenies of some abelian group.
\begin{itemize}
\item
The \emph{katakernel} of $\Gamma$ is $\Kat(\Gamma) = \sum_{\gamma\in\Gamma}\kat\gamma$.
\item
The \emph{bi-katakernel} of $(\Gamma, \Delta)$ is $\Kat(\Gamma,\Delta) = \Kat(\Gamma) + \Kat(\Delta)$.
\end{itemize}
\end{definition}

\begin{lemma}[Katakernel Lemma]\label{l:radical}
Let $A$ be an abelian group and $\Gamma, \Delta$ be two sharply commuting prerings of endogenies.
\begin{enumerate}%[label=(\roman*)]
\item
The katakernel of $\Gamma$ is fully $\Gamma$-invariant.
In particular $\Gamma$ acts on $A/\Kat(\Gamma)$ by endomorphisms.
\item
The bi-katakernel of $(\Gamma, \Delta)$ is fully $\Gamma$- and $\Delta$-invariant.
Hence $\Gamma$ and $\Delta$ act on $A/\Kat(\Gamma,\Delta)$ by commuting endomorphisms.
\item\label{i:inverse}
For any $\gamma\in\Gamma$ the subgroup $\gamma^{-1}[\Kat(\Gamma,\Delta)]$ is fully $\Delta$-invariant.
\end{enumerate}
\end{lemma}
\begin{proof}\leavevmode
\begin{enumerate}%[label=(\roman*)]
\item
Let $\gamma,\gamma'\in\Gamma$. Then $\gamma[\kat\gamma'] = \gamma\gamma'[0] = \kat(\gamma\gamma') \le \Kat(\Gamma)$. This shows full $\Gamma$-invariance. Clearly, factoring out the katakernels transforms the endogenies into endomorphisms (which would not hold under \emph{weak} invariance).
\item
For $\gamma\in\Gamma$ its katakernel $\kat\gamma$ is \emph{weakly} $\Delta$-invariant by sharp commutation and Lemma~\ref{l:invariance}.\ref{i:cokerinvariance:reformulation}. So $\delta[\kat\gamma] \le \kat\gamma + \kat\delta \le \Kat(\Gamma,\Delta)$. This proves \emph{full} invariance.
\item
Put $K = \gamma^{-1}[\Kat(\Gamma,\Delta)]$, and note that $K\ge\Kat(\Gamma,\Delta)$. Consider any $\delta\in\Delta$. Then 
\[\gamma\delta[K] \le \delta\gamma[K] + \kat\gamma + \kat\delta \le \delta[\Kat(\Gamma,\Delta)] + \Kat(\Gamma,\Delta) = \Kat(\Gamma,\Delta),\]
so $\delta[K] \le \gamma^{-1}[\Kat(\Gamma,\Delta)] = K$ and $K$ is fully $\Delta$-invariant.
\qedhere
\end{enumerate}
\end{proof}

\begin{remark}
The proof of~\ref{i:inverse}.\ even shows that for any fully $\Delta$-invariant subgroup $B$ containing $\Kat(\Delta)$ and $\kat\gamma$, the inverse image $\gamma^{-1}[B]$ is fully $\Delta$-invariant.
\end{remark}

\subsection{Definability and Dimensionality}\label{s:dimensionality}
We now add some definability on $A$; an endogeny is said to be \emph{definable} if it is, as a relation.
The set of definable endogenies $\DefEndog(A)$ of $A$ is a sub-prering of $\Endog(A)$ (see~Section~\ref{s:prering}), viz., it is closed under $ +$, $-$ and $\circ$.

Our endogeny prerings will in general not be definable, but merely invariant in the following sense.

\begin{definition}
Let $A$ be a definable abelian group, and $\Gamma$ a ring of definable endogenies of $A$. We say that $\Gamma$ is {\em invariant} if there is a set $C$ of parameters such that (in a big $\kappa$-saturated monster model) $\Gamma$ is stabilized by all automorphisms which fix $C$.\end{definition}

\begin{remark} Thus $\Gamma$ is invariant if the set of parameters for the endogenies in $\Gamma$ is a union of types over $C$.\end{remark}

There is an unfortunate clash of terminology between the model-theoretic and the algebraic meaning of invariance, but both are well-established. This should not lead to confusion.

We shall now define finite-dimensionality, introduced in \cite{wa20}.

\begin{definition}\label{d:dimensionality} A theory $T$ is {\em finite-dimensional} if there is a dimension function $\dim$ from the collection of all interpretable sets in models of $T$ to $\mathbb N\cup\{\pm\infty\}$, 
satisfying for a formula $\varphi(x,y)$ and interpretable sets $X$ and $Y$:\begin{itemize}
\item{\em Invariance:} If $a\equiv a'$ then $\dim(\varphi(x,a))=\dim(\varphi(x,a'))$.
\item{\em Algebraicity:} $X$ is finite non-empty if and only if $\dim(X)=0$, and $\dim(\emptyset)=-\infty$.
\item{\em Union:} $\dim(X\cup Y)=\max\{\dim(X),\dim(Y)\}$.
\item{\em Fibration:} If $f:X\to Y$ is a interpretable map such that $\dim(f^{-1}(y))=d$ for all $y\in Y$, then $\dim(X)=\dim(Y)+d$.
\end{itemize}
\end{definition}

Examples of finite-dimensional theories include groups of finite Morley rank (where dimension is Morley rank), groups of finite SU-rank (where dimension is SU-rank), and groups definable in $o$-minimal theories (where dimension is $o$-minimal dimension).

The following will be used without mention and follows immediately from the fibration axiom.

\begin{lemma}[Dimension Lemma]
In a finite-dimensional theory, let $A$ be a definable group and $\gamma$ be a definable endogeny of $A$. Then $\dim \ker \gamma + \dim \im \gamma = \dim A$.
\end{lemma}
\begin{proof}
Treat $\gamma$ as a subgroup of $A^2$. The first projection has image $A$ and kernel isomorphic to $\kat \gamma$, so $\dim \gamma = \dim A$. Now the second projection has image $\im \gamma$ and kernel $\ker \gamma$, so $\dim \gamma = \dim \ker \gamma + \dim \im \gamma$.
\end{proof}

In particular, $|A:\im\gamma|$ is finite if and only if $\ker\gamma$ is finite.

\subsection{Connectedness and minimality}\label{s:connectedness}

The notion of connectedness is the usual one. Having a connected component is equivalent to being virtually connected. 

\begin{lemma}[Connectedness Lemma]\label{l:connectedness}
Let $B\le A$ be a virtually connected subgroup, and $\gamma$ an endogeny of $A$. Then $\gamma[B]$ is virtually connected. Moreover, $\gamma[B^\circ]=\gamma[B]^\circ+\kat\gamma$.
\end{lemma}
\begin{proof}%\leavevmode
Let $C\le\gamma[B]$ have finite index in $\gamma[B]$. Then $\gamma^{-1}[C]$ has finite index in $B$ and contains $B^\circ$. Considering the induced endomorphism $\bar\gamma:B\to\gamma[B]/\kat\gamma$, we see that 
\begin{align*}
|\gamma[B]:C+\kat\gamma| & =|\gamma[B]/\kat\gamma:(C+\kat\gamma)/\kat\gamma|\\
& \le |B:\gamma^{-1}[C+\kat\gamma]|=|B:\gamma^{-1}[C]|\le|B:B^\circ|.
\end{align*}
So the index of any definable subgroup of finite index in $\gamma[B]$ is bounded, and $\gamma[B]$ is virtually connected.

Clearly $\kat\gamma\le\gamma[B^\circ]$, and $\gamma[B^\circ]$ has finite index in $\gamma[B]$. Thus $\gamma[B]^\circ+\kat\gamma\le\gamma[B^\circ]$. Conversely, $\gamma^{-1}[\gamma[B]^\circ]$ has finite index in $B$ and contains $B^\circ$. Hence $\gamma[B^\circ]\le\gamma[B]^\circ+\kat\gamma$, and equality holds.
\end{proof}

\begin{remarks}\label{r:connectedness}\leavevmode
\begin{itemize}
\item
In general $\ker \gamma$ need not have a connected component. Moreover, if $\delta$ commutes sharply with $\gamma$, then $\ker \gamma$ need not be weakly $\delta$-invariant (see Remark \ref{r:inverse}.). But one can prove that \emph{if} $\ker \gamma$ is virtually connected, then $(\ker \gamma)^\circ$ is weakly $\delta$-invariant.
\item
By methods similar to the proof of Lemma~\ref{l:connectedness}, one can show that if $A$ is virtually connected, %-by-finite, 
then so is every definable endogeny of $A$, as a subgroup of $A^2$. Thus if $A$ is connected, every endogeny is equivalent to a \emph{connected endogeny}. However these do not form a sub-prering of $\Endog(A)$, and $C^\sharp(\delta)$ is not closed under taking connected components. We therefore do not pursue the matter any further.
\end{itemize}
\end{remarks}

We can now define the relevant notion of minimality.
\begin{definition} Let $A$ be a virtually connected definable abelian group in a finite-dimensional theory, and $\Gamma$, $\Delta$ two invariant rings of definable endogenies of $A$. We say that $A$ is $(\Gamma,\Delta)$-minimal if there is no infinite, definable, virtually connected, weakly $\Gamma$- and $\Delta$-invariant subgroup of infinite index.\end{definition}

\begin{remark} As noted in Remark \ref{r:subcommuting}, sharp commutation is not hereditary on weakly $\Gamma$- and $\Delta$-invariant subgroups. Thus, if $A$ is a virtually connected definable abelian group in a finite-dimensional theory, and $\Gamma$, $\Delta$ two sharply commuting invariant rings of definable endogenies of $A$, then $A$ contains a $(\Gamma,\Delta)$-minimal virtually connected subgroup $B$ by finite-dimensionality, but need not contain one where the restriction-corestrictions $\rho_B(\Gamma)$ and $\rho_B(\Delta)$ commute sharply.

In particular, since $A$ is virtually connected, its connected component $A^\circ$ will be weakly $\Gamma$- and $\Delta$-invariant, but again $\rho_{A^\circ}(\Gamma)$ and $\rho_{A^\circ}(\Delta)$ might not commute sharply. For this reason we work in the virtually connected category rather than the more usual connected one.\end{remark}

\section{Proof of Theorem~A}\label{S:radicalisation}

In this section we shall prove our first main theorem. Full familiarity with the notions of Section~\ref{S:endogenies} is required. We can now give the precise statement.

\begin{theorema}
In a finite-dimensional theory, let $A$ be a definable, virtually connected, abelian group. Let $\Gamma$ and $\Delta$ be two invariant prerings of definable endogenies of $A$ such that $A$ is $(\Gamma,\Delta)$-minimal. Suppose $\Gamma$ and $\Delta$ commute sharply, both are essentially infinite, and one of them is essentially unbounded.
Then $\Kat(\Gamma,\Delta)$ is finite, and the induced action of $\Gamma$ and $\Delta$ on $A/\Kat(\Gamma,\Delta)$ is by commuting endomorphisms.\end{theorema}

Of course, the finite group $F$ from the introduction equals $\Kat(\Gamma,\Delta)$. The proof is by induction on dimension, reducing to a local form given in~Section~\ref{s:radicalisation:local}. The reduction itself involves subgroups of $A$ called \emph{lines} and introduced in Section~\ref{s:radicalisation:lines}.
One needs to reduce $\Gamma$ and $\Delta$ to the relevant variants $\Gamma_L, \Delta_L$ in Section~\ref{s:localendogenies}.
Returning from lines to $A$ takes place in Section~\ref{s:radicalisation:globalisation}. Then the inductive proof of Theorem~A quickly follows in Section~\ref{s:induction}.

One general notation is used throughout this section: when $B \leq A$ are (abelian) groups, we write $A \gg B$ if $A/B$ is infinite. In particular, $A \gg 0$ means that $A$ is infinite.

\subsection{Local version (finite kernels)}\label{s:radicalisation:local}

\begin{proposition}[Local version of Theorem A]\label{t:local}
Under the same hypotheses as Theorem~A, instead of assuming $(\Gamma, \Delta)$-minimality, suppose that every endogeny in $\Gamma \cup \Delta$ inequivalent to $0$ has finite kernel. Then $\Kat(\Gamma,\Delta)$ is finite.
\end{proposition}
\begin{proof}\leavevmode
First suppose that $\Gamma$ is essentially unbounded. If the sum giving $\Kat(\Delta)$ is infinite, we can take a subsum $B$ which is countably infinite. Note that it is weakly $\Gamma$-invariant.
Of course $\Endog(B)$ is bounded.
By essential unboundedness, there are inequivalent $\gamma_1, \gamma_2 \in \Gamma$ restricting to the same $\rho_B(\gamma_i)$, meaning that $B \leq \ker (\gamma_1 - \gamma_2)$, a contradiction to kernel finiteness. Therefore $\Kat(\Delta)$ is finite.

Now suppose that $\Gamma$ is essentially bounded. By assumption, $\Delta$ is not; so by the above, $\Kat(\Gamma)$ is finite. Now $\kat\gamma \le \Kat(\Gamma)$ for each $\gamma \in \Gamma$, so there are only finitely many possible katakernels for elements of $\Gamma$. In particular, there are finitely many $\gamma' \in \Gamma$ equivalent to each $\gamma$.
Thus equivalence classes are finite, and there are only boundedly many of them by essential boundedness.
This proves that $\Gamma$ is bounded, as a subset of $\Endog(A)$.

It is however essentially infinite. So for every $\delta\in\Delta$ there is $\gamma \in \Gamma$ inequivalent to $0$ which annihilates the finite weakly $\Gamma$-invariant group $\kat \delta$, viz., $\kat \delta \leq \ker \gamma$. This proves that $\Kat(\Delta)\leq \sum \{\ker \gamma: \gamma \in \Gamma\wedge \gamma \not\sim 0\}$. But the latter is bounded, and so is $\Kat(\Delta)$, which is fully $\Delta$-invariant. Since $\Delta$ is essentially unbounded, the argument from the first case implies finiteness of $\Kat(\Delta)$.

By symmetry, $\Kat(\Gamma)$ is finite, as is $\Kat(\Gamma,\Delta)$.
\end{proof}

\subsection{Lines}\label{s:radicalisation:lines}

The proof of Theorem~A starts here; $A$, $\Gamma$ and $\Delta$ are as in the statement until the end of Section~\ref{S:radicalisation}.

The idea behind the reduction is simple: we consider minimal infinite $\Gamma$-images. Formally:

\begin{definition}
A \emph{line} is a $\Gamma$-image $\gamma[A]$ which is infinite and minimal with respect to inclusion,
% A $\Gamma$-image $\gamma[A]$ is a {\em line} if it is minimal infinite,
i.e.~for all $\gamma'\in\Gamma$, if $\gamma'[A]\le\gamma[A]$ is infinite, then $\gamma'[A]=\gamma[A]$.
We denote the set of lines by $\Lambda$.
\end{definition}

By finite-dimensionality and virtual connectedness, lines do exist (also see~\ref{i:Lexistence} in the following proposition).

\begin{proposition}\label{p:lines}\leavevmode
\begin{enumerate}%[label=(\arabic*)]%], series=claims]
\item\label{i:Linvariance}  Any line is weakly $\Delta$-invariant and virtually connected.
\item\label{i:Lexistence}
If $\gamma[A] \gg 0$, then there is a line $\gamma'[A]\le\gamma[A]$.
\item\label{i:Asumoflines} Let $L$ be a line. Then
$A^\circ$ is contained in a finite sum of $\Gamma$-images of $L$.
\item\label{i:transitive}
% The action of $\Gamma$ on $\Lambda$ is transitive. Better: 
If $L$ and $L'= \gamma'[A]$ are lines, then there is $\gamma^* \in \Gamma$ such that $L'=\gamma'\gamma^*[L]$.
\item\label{i:finitekernel} All lines have the same dimension. If $L$ is a line and  $\gamma[L] \gg 0$ for some $\gamma\in\Gamma$, then $L\cap\ker\gamma$ is finite.
\item\label{i:setidempotent} If $L$ is a line, there is $\gamma \in \Gamma$ such that $\gamma[A] = \gamma[L]=L$.
Then $\check\gamma = \rho_{L}(\gamma)$ is surjective with finite kernel.
\end{enumerate}
Items~\ref{i:Linvariance}.\ and \ref{i:Lexistence}.\ do \emph{not} need $(\Gamma, \Delta)$-minimality of $A$.
\end{proposition}

\begin{proof}\leavevmode\begin{enumerate}%[label={\itshape (\arabic*)}]%, series=proofs]
\item This holds for all $\Gamma$-images by Lemmas \ref{l:invariance} and \ref{l:connectedness}.
Note that the connected component of a line might not be a line itself since it need not be a $\Gamma$-image of $A$.
\item Since $\Gamma$-images are virtually connected, any proper subimage has either smaller dimension or smaller index over the common connected component. We hence have a descending chain condition on $\Gamma$-images, and any infinite one contains a minimal infinite $\Gamma$-image, i.e.\ a line.
\item By Lemma~\ref{l:connectedness}., for all $\gamma\in\Gamma$ the image $\gamma[L]$ is virtually connected, and $\gamma[L]^\circ$ is weakly $\Delta$-invariant. Then $A' = \sum_{\gamma\in\Gamma} \gamma[L]^\circ \gg 0$ is well-defined, definable, connected, and weakly $\Delta$-invariant; moreover it is a finite sum $A' = \sum_{i<n}\gamma_i[L]^\circ$.

We contend that $A'$ is weakly $\Gamma$-invariant.
Consider $\gamma\in \Gamma$. Then
\[\gamma\gamma_i[L]^\circ\le \gamma[\gamma_i[L]^\circ] \leq \gamma\gamma_i[L],\]
so by Lemma~\ref{l:connectedness} one gets $\gamma[\gamma_i[L]^\circ] \leq (\gamma\gamma_i[L])^\circ + \kat \gamma \leq A' + \kat \gamma$.
%If $\gamma\gamma_i[L]$ is not a line, it is finite by~\ref{i:finitekernel}.; in either case, $\gamma\gamma_i[L]^\circ \leq A'$. So $\gamma[\gamma_i[L]^\circ] \leq A' + \kat \gamma$.
Summing over $i<n$, we find $\gamma[A'] \leq A' + \kat \gamma$.

By $(\Gamma, \Delta)$-minimality of $A$, we have $A^\circ \le A' \le \sum_{i<n} \gamma_i[L]$, as required. 
\item
Let $L, L'$ % \in \Lambda$ be given
be lines, with $L=\gamma[A]$ and $L' = \gamma'[A]$, and consider $\gamma_i\in\Gamma$ for $i<n$ with $A^\circ\le\sum_{i<n}\gamma_i[L]^\circ$ as in \ref{i:Asumoflines}. Then there is $i<n$ such that $\gamma'\gamma_i [L] \gg 0$. So 
$$0 \ll \gamma'\gamma_i\gamma[A]=\gamma'\gamma_i[L] \leq \gamma'[A]=L'.$$
By minimality, we have equality and $\gamma'\gamma_i[L]=L'$.
\item If $L$, $L'$ are lines and $\gamma'\in\Gamma$ with $L'=\gamma'[L]$, then $\dim L'\le\dim L$. By symmetry%and transitivity of the action
, $\dim L\le\dim L'$ and we have equality. But then, for any $\gamma\in\Gamma$ such that $\gamma[L]$ is infinite, $\gamma[L]$ contains a line $L''$, so
$$\dim L=\dim\gamma[L]+\dim(L\cap\ker\gamma)\ge\dim L''+\dim(L\cap\ker\gamma)=\dim L+\dim(L\cap\ker\gamma).$$
It follows that $\dim(L\cap\ker\gamma)=0$ and $L\cap\ker\gamma$ is finite. Note that $\gamma[L]$ is a finite extension of $L''$, but need not be a line itself.
\item
We apply \ref{i:transitive}.\ with $L = L'=\gamma'[A]$. So there is $\gamma^*\in\Gamma$ with $L=\gamma'\gamma^*[L]$. Put $\gamma=\gamma'\gamma^*$. Then
$$L=\gamma[L]\le\gamma[A]=\gamma'\gamma^*[A]\le\gamma'[A]=L.$$
Thus equality holds. The rest follows from \ref{i:finitekernel}.
\qedhere
\end{enumerate}
\end{proof}

\subsection{Local endogenies}\label{s:localendogenies}

Reduction to lines involves finding the suitable local analogues of $\Gamma$ and $\Delta$. We want two prerings sharply centralising each other, each essentially as large as its global analogue (where \emph{large} means either infinite or unbounded).

\begin{definition}
Let $L \in \Lambda$ be a line, and recall that for an endogeny $\varphi \in \Endog(A)$ leaving $L$ weakly invariant, $\check\varphi = \rho_L(\varphi)$ is its restriction-corestriction to $L$. We put
$$\Gamma_L = \langle\check\gamma: \gamma \in \Gamma\wedge \im \gamma \leq L\rangle \le \Endog(L)\quad\mbox{and}\quad\Delta_L = \langle\check\delta: \delta \in \Delta\rangle \le \Endog(L).$$
\end{definition}

\begin{proposition}\label{p:GammaLDeltaL}
Let $L \in \Lambda$ be any line.
\begin{enumerate}%[label=(\arabic*)]
\item\label{i:equalequivalent} Any $\gamma'\in\Gamma_L$ is equal to $\check\gamma$ for some $\gamma\in \Gamma$ with $\im\gamma\le L$, and any $\delta'\in\Delta_L$ is equivalent to $\check\delta$ for some $\delta\in\Delta$.
\item\label{i:GammaLDeltaLcommute}
$\Gamma_L$ and $\Delta_L$ commute sharply.
\item\label{i:lineminimal}
$L$ is $(\Gamma_L, \Delta_L)$-minimal.
\item\label{i:Gammalarge}
$\Gamma_L$ is (at least) as essentially large as $\Gamma$, and $\Delta_L$ is (at least) as essentially large as $\Delta$. 
\end{enumerate}
Items~\ref{i:equalequivalent}.\ and \ref{i:GammaLDeltaLcommute}.\ do \emph{not} use $(\Gamma, \Delta)$-minimality.
\end{proposition}

\begin{proof}
Say $L = \gamma_0[A]$.
\begin{enumerate}%[label=(\arabic*)]
% \item Follows from \ref{i:finitekernel}.
\item If $\gamma_1,\gamma_2\in\Gamma$ with $\im\gamma_1\le L$ and $\im\gamma_2\le L$, then $\rho_L(\gamma_1\pm\gamma_2)=\rho_L(\gamma_1)\pm\rho_L(\gamma_2)$ and $\rho_L(\gamma_1\circ\gamma_2)=\rho_L(\gamma_1)\circ\rho_L(\gamma_2)$.

Since $L$ is weakly $\Delta$-invariant, $\rho_L(\delta_1\pm\delta_2)$ is equivalent to $\rho_L(\delta_1)\pm\rho_L(\delta_2)$, and $\rho_L(\delta_1\circ\delta_2)$ is equivalent to $\rho_L(\delta_1)\circ\rho_L(\delta_2)$, for any $\delta_1,\delta_2\in\Delta$.
\item
%Clearly $\check\Delta_L \subseteq \CC(\CC(\check\Delta_L)) = \Delta_L$.
For $\gamma \in\Gamma$ with $\im \gamma \leq L$, one has $\kat \gamma = \kat \check\gamma \leq L$, so for $\delta \in \Delta$ one finds:
$$%\begin{align*}
(\check\gamma \check\delta - \check\delta\check\gamma)[L] \leq 
L\cap(\gamma\delta-\delta\gamma)[A]\le
L \cap (\kat \gamma + \kat \delta)
= \kat \gamma + (L \cap \kat \delta)
= \kat\check\gamma + \kat\check\delta.
$$%\end{align*}
Therefore $\check\gamma$ commutes sharply with $\check\delta$ for all $\delta\in\Delta$. By Lemma \ref{l:cokerinvariant}.\ref{i:commutationprering}, it commutes sharply with $\Delta_L$, and then $\Delta_L$ commutes sharply with $\Gamma_L$.
%But sharp centralisers are sub-prerings by Lemma~\ref{l:almostaring}, so $\check\gamma$ commutes sharply with all of $\Delta_L$.
%\item
%The second is by construction. Now by~\ref{i:GammaLDeltaLcommute}, $\check\Delta_L \subseteq \Delta_L$, so $\Gamma_L \leq \CC(\Delta_L) \leq \CC(\check\Delta_L) = \Gamma_L$.
\item
%We even prove $(\check\Gamma_L, \check\Delta_L)$-minimality. By~\ref{i:GammaLDeltaLcommute}, this will suffice.
%
Let $0 \ll B\le L$ be virtually connected and definable. By finite-dimensionality, there are finitely many $\gamma_i\in\Gamma$ and $\delta_i\in\Delta$ such that $S = \sum_i\gamma_i\delta_i[B]^\circ \gg 0$ has maximal dimension possible. Then for any $\varphi\in\Gamma\cup\Delta$ one has $\varphi[S]^\circ\le S$, so $\varphi[S]\le S + \kat\varphi$ by Lemma~\ref{l:connectedness}, and $S$ is weakly $(\Gamma, \Delta)$-invariant. Hence $S$ has finite index in $A$ by $(\Gamma, \Delta)$-minimality of $A$. But then 
\[\gamma_0[S] \leq \sum_{i<n}\gamma_0\gamma_i\delta_i[B] = 
\sum_{i<n}\rho_L(\gamma_0\gamma_i)\check\delta_i[B] + \sum_{i<n}\gamma_0\gamma_i[\kat\delta_i]\]
has finite index in $\gamma_0[A]=L$, as does $\sum_i\rho_L(\gamma_0\gamma_i)\check\delta_i[B]$.
Now if $B$ is weakly $(\Gamma_L,\Delta_L)$-invariant, then it has finite index in $L$, as $\rho_L(\gamma_0\gamma_i)\in\Gamma_L$ and $\check\delta_i\in\Delta_L$. Thus $L$ is $(\Gamma_L,\Delta_L)$-minimal.
%This is our definition of $(\check\Gamma_L, \check\Delta_L)$-minimality.
\item
%By~\ref{i:GammaLDeltaLcommute}, it is enough to prove essential largeness of $\check\Gamma_L$ and $\check\Delta_L$.
%
We start with $\Delta_L$ and prove that the restriction-corestriction map $\rho_L\colon \Delta\to \check\Delta_L$ is bijective up to equivalence in $\Endog(A)$, resp.~$\Endog(L)$.
But up to equivalence $\rho_L$ induces a morphism, so we simply study the fibre over $0$.
Suppose $\delta\in\Delta$ with $\check\delta = \rho_L(\delta)$ equivalent to $0$. Then for $\gamma \in \Gamma$ we have
\[\delta\gamma[L] \le \gamma\delta[L] + \kat\gamma + \kat\delta \le \gamma[\check\delta[L] + \kat\delta] + \kat\gamma + \kat\delta,\]
which is finite. Hence $\delta[A]$ is finite by Proposition~\ref{p:lines}.\ref{i:Asumoflines}, so $\delta$ is equivalent to $0$.
% This proves that essential largeness is preserved from $\Delta$ to $\check\Delta_L \subseteq \Delta_L$.

The rest of the proof is devoted to $\Gamma_L$.
Consider a finite subset $\{\gamma_1, \dots, \gamma_n\} \subseteq \Gamma$ such that all $\gamma_i[A]$ are lines, and $K = \bigcap_i \ker \gamma_i$ has minimal dimension possible.

We claim that one of the $\gamma_i \Gamma$ is essentially large.
Suppose not. Argueing if necessary in the quotient ring $\Gamma/\!\sim$ modulo equivalence, there is $\varphi \in \Gamma$ inequivalent to $0$ such that all $\gamma_i \varphi$ are equivalent to $0$. This means that each $\gamma_i \varphi[A]$ is finite, viz., $\varphi[A]^\circ \leq \bigcap \ker \gamma_i = K$.
However $\varphi[A]$ is infinite; by Proposition~\ref{p:lines}.\ref{i:Lexistence} and \ref{p:lines}.\ref{i:setidempotent}, there is a line $L'=\gamma'[L']=\gamma'[A]$ contained in $\varphi[A]$.
But $\dim (K \cap \ker \gamma') = \dim K$ by minimality of $\dim K$, so $L'^\circ\leq\varphi[A]^\circ \leq K \cap \ker \gamma$, contradicting $L' = \gamma[L'] \gg 0$.
Hence there is $i$ with $\gamma_i\Gamma$ essentially large.

We shall deduce that $\gamma_0 \Gamma$ is essentially large (recall $L = \gamma_0[A]$). Since $L_i = \gamma_i[A]$ is a line, by Proposition~\ref{p:lines}.\ref{i:transitive} there is $\alpha\in\Gamma$ such that $\gamma_0\alpha[L_i]=L$. By Proposition~\ref{p:lines}.\ref{i:finitekernel}, $L_i \cap \ker (\gamma_0\alpha)$ is finite.
Now if for some endogeny $\varphi$ one has $\gamma_0\alpha \gamma_i\varphi\sim 0$, then $\gamma_i \varphi[A]^\circ \leq \ker (\gamma_0\alpha)$. But $L_i \cap \ker (\gamma_0\alpha)$ is finite, so $\gamma_i\varphi[A]^\circ = 0$ and $\gamma_i\varphi\sim 0$.
It follows that left multiplication $\gamma_i \Gamma \to \gamma_0 \alpha \gamma_i \Gamma$ preserves inequivalence in $\Endog(A)$, so $\gamma_0 \alpha \gamma_i \Gamma \le \gamma_0\Gamma$ is essentially (at least) as large as $\gamma_i\Gamma$.

We now conclude that $\Gamma_L$ is essentially large.
By Proposition~\ref{p:lines}~\ref{i:Asumoflines}, there are $\gamma_1, \dots, \gamma_n \in \Gamma$ (not necessarily the same as above) such that $A^\circ \le \sum_i \gamma_i[L]$. Let $J$ be a large set indexing a family $\{\varphi_j:j\in J\} \subseteq \gamma_0\Gamma$ of pairwise inequivalent endogenies.
Notice that $\im \varphi_j \gamma_i \leq L$, so $\rho_L(\varphi_j \gamma_i) \in \Gamma_L$.
Then for $j \neq j'$ in $J$ there is $i \in \{1, \dots, n\}$ such that $(\varphi_j-\varphi_{j'})\gamma_i[L] \gg 0$. By Ramsey or Erd\H{o}s-Rado (or, alternatively, Ramsey and compactness) there are a common $i \in \{1, \dots, n\}$ and a large $J_0\subseteq J$ such that for any $j \neq j'$ in $J_0$ we have $(\varphi_j-\varphi_{j'})\gamma_i[L] \gg 0$. Then $\{\rho_L(\varphi_j\gamma_i):j\in J_0\}$ are all inequivalent and in $\Gamma_L$, so $\Gamma_L$ is essentially large.
\qedhere
\end{enumerate}
\end{proof}

\subsection{Controlling the katakernel}\label{s:radicalisation:globalisation}

\begin{definition}
A \emph{quasi-projection} of $A$ onto $B \leq A$ is an endogeny $\pi$ with image $\pi[A] = B$ and such that $\rho_B(\pi) = 1_B + \kat \pi$.
\end{definition}

\begin{proposition}\label{p:linereduction2}
Suppose that $\Gamma=\CC(\Delta)$, and that $\Kat(\Gamma_L,\Delta_L)$ is finite for every line $L$.
\begin{enumerate}%[label=(\arabic*)]
\item\label{i:pi}
For each line $L$ there is a quasi-projection $\pi\in\Gamma$ from $A$ to $L$ with katakernel $\Kat(\Gamma_L,\Delta_L)$.
\item\label{i:Aquasidirectsum}
There are lines $L_0, \dots, L_n$ and quasi-projections $\pi_i:A\to L_i$ in $\Gamma$ for $i\le n$ such that $\sum_{i\le n}\pi_i$ is equivalent to $1$ in $\Gamma$.
\item\label{i:radAfini}
$\Kat(\Gamma, \Delta)$ is finite.
\end{enumerate}
\end{proposition}

\begin{proof}
We shall simply write $\Kat(L)$ for $\Kat(\Gamma_L,\Delta_L)$, and $\Kat(A)$ for $\Kat(\Gamma,\Delta)$.
\begin{enumerate}%[label=(\arabic*)]
\item
By Proposition~\ref{p:lines}.\ref{i:setidempotent}, there is $\gamma_0 \in \Gamma$ such that $L = \gamma_0[L] = \gamma_0[A]$.
In particular, $A = L + \ker \gamma_0 = L + \gamma_0^{-1}[\kat \gamma_0]$.
Moreover, $\check\gamma_0 = \rho_L(\gamma_0)$ is surjective with a finite kernel. We apply Lemma~\ref{l:radical} to $L, \Gamma_L, \Delta_L$. Then $\check\gamma_0[\Kat(L)] \leq \Kat(L)$, and $K_L = \check\gamma_0^{-1}[\Kat(L)] \ge \Kat(L)$. Moreover $K_L$ is fully $\Delta_L$-invariant by Lemma~\ref{l:radical}.\ref{i:inverse}.; since $\ker\check\gamma_0$ is finite and $\Kat(L)$ is finite by assumption, $K_L$ must be finite.

For $a \in A$, let $\pi[a] = \check\gamma_0^{-1}\gamma_0[a] + K_L$, an endogeny from $A$ onto $L$ with katakernel $K_L$.

We shall show $\pi \in \CC(\Delta)$. Let $\delta \in \Delta$. Then by sharp commutation of $\gamma$ and $\delta$, we get:
\begin{align*}
\check\gamma_0\pi\delta[a] &= \check\gamma_0[\check\gamma_0^{-1}\gamma_0\delta[a] + K_L] = \gamma_0\delta[a] + \check\gamma_0[K_L] \le \gamma_0\delta[a] + \Kat(L)\\
&  \leq L \cap (\delta\gamma_0[a] + \kat \gamma_0) + \Kat(L) = (L \cap \delta\gamma_0[a]) + \Kat(L).
\end{align*}
Since $\gamma_0[A] = L = \gamma_0[L]$, there is $\ell \in L$ with $\gamma_0[a] = \gamma_0[\ell] = \check\gamma_0[\ell]$. Thus $\pi[a] = \ell + K_L$ and:
\[
\check\gamma_0\pi\delta[a]
\le (L\cap\delta\gamma_0[a]) + \Kat(L)
= (L\cap\delta\gamma_0[\ell]) + \Kat(L) = \check\delta\check\gamma_0[\ell] + \Kat(L).\]
On the other hand, by Proposition~\ref{p:GammaLDeltaL}.\ref{i:GammaLDeltaLcommute}, $\check\gamma_0$ and $\check\delta$ commute sharply, so:
\[\check\gamma_0\check\delta\pi[a] = \check\gamma_0\check\delta
[\ell + K_L] \le \check\delta\check\gamma_0[\ell + K_L] + \kat\check\delta + \kat\check\gamma_0 + \Kat(L) = \check\delta\check\gamma_0[\ell] + \Kat(L).
\]
Together, this gives $\check\gamma_0 (\pi\delta - \check\delta\pi)[a] \leq \Kat(L)$.
Since $K_L = \check\gamma_0^{-1}[\Kat(L)] \ge \Kat(L)$, we conclude
\[\pi\delta[a] + K_L = \check\delta\pi[a] + K_L,\]
whence
\[\delta\pi[a]-\pi\delta[a] = \check\delta\pi[a]-\pi\delta[a] + \kat\delta\le K_L + \kat\delta = \kat\pi + \kat\delta.\]
Thus $\pi\in\CC(\Delta) = \Gamma$.
In particular $\check\pi \in \Gamma_L$ and $\Kat(L) \le K_L = \kat\check\pi \le \Kat(L)$.
\item We shall show first that if $\gamma\in\Gamma$ is such that $B=\gamma[A]$ has infinite index in $A$ and $\rho_B(\gamma)$ is equivalent to $1_B$, then there exist a line $L$ and a quasi-projection $\pi:A\to L$ such that $B\cap L$ is finite and $\rho_{B+L}(\gamma+\pi)$ is equivalent to $1_{B+L}$.

So suppose that $\gamma$, $B$ are as given, and $F\le B$ is a finite group with $a-\gamma[a]\le F$ for all $a\in B$.
If $a\in(1-\gamma)[A]$, say 
$$a\in(1-\gamma)[b] = b-\gamma[b]=b-b'+\kat(\gamma)$$
for some $b\in A$ and $b'\in\gamma[b] \le B$, we have:
\begin{align*}
(1-\gamma)[a] & = a-\gamma[a]\le b-b' +\kat(\gamma)-\gamma[b-b' + \kat(\gamma)]\\
& = b-b' + \kat(\gamma)-\gamma[b] + \gamma[b'] + \gamma[\kat(\gamma)]\\
& = b-b'-b' + b' + \kat(\gamma^2) = b-b' + \kat(\gamma^2) = a + \kat(\gamma^2).
\end{align*}

Put $H = \im(1-\gamma)$, a weakly $\Delta$-invariant virtually connected subgroup of $A$. Then $A = B+H$, and $\rho_H(1-\gamma)$ is equivalent to $1_H$. If $a\in B\cap H$, then 
\[\gamma[a] =(1-(1-\gamma))[a]\le a-(a+\kat(\gamma^2))=\kat(\gamma^2),\]
and $a\in \rho_B(\gamma)^{-1}[\kat(\gamma^2)]$. But $\dim B=\dim\gamma[B]$, so $\dim\ker\rho_B(\gamma)=0$ and $\ker\rho_B(\gamma)$ is finite. Thus $B\cap H$ is finite and the sum is quasi-direct.

Since $B$ has infinite index in $A$, the quasi-complement $H=\im(1-\gamma)$ is infinite and contains a line $L$ by Proposition~\ref{p:lines}.\ref{i:Lexistence}; by part~\ref{i:pi}.\ there is a quasi-projection $\pi'\in\Gamma$ onto $L$ with katakernel $\kat(\pi')=\Kat(L)$. But $1-\gamma$ is equivalent to $0$ on $B$ and to $1$ on $H$, so $\pi = \pi'(1-\gamma)$ is also a quasi-projection to $L$ which is equivalent to $0$ on $B$ and to $1$ on $L$. Moreover, 
$$\Kat(L)=\kat(\pi')\le\kat(\pi)\le\Kat(L),$$
so we have equality. As $\gamma$ is equivalent to $0$ on $L$, we get that $\gamma+\pi$ is equivalent to $1$ on $B+L$. Note that $\im(\gamma+\pi)+(B\cap L)=B+L$, but $\im(\gamma+\pi)$ could be a subgroup of finite index in $B+L$.

Starting with $\gamma_0=0$, we now construct inductively lines $L_i$ and quasi-projections $\pi_i\in\Gamma$ onto $L_i$ with $\kat(\pi_i)=\Kat(L_i)$
such that $L_i\cap\sum_{j<i}L_j$ is finite for all $j$ and $\sum_{j\le i}\pi_i$ is equivalent to $1$ on $\sum_{j\le i}L_i$. By finite-dimensionality, $A^\circ\le\sum_{i\le n}L_i$ for some $n<\omega$. Then $A^\circ$ is contained in a quasi-direct sum of lines, and $\sum_{i\le n}\pi_i$ is equivalent to the identity on $\sum_{i\le n}L_i$. Thus $(1-\sum_{i\le n}\pi_i)[A]$ is finite, and $\sum_{i\le n}\pi_i$ is equivalent to $1$ on $A$.
\item
Work with a system of lines and quasi-projections as above.
Put $F = \im(1-\sum_i\pi_i)$, a finite weakly $\Delta$-invariant group containing all $\Kat(L_i)$.

We contend $\Kat (A) \leq F$.
For $\gamma\in\Gamma$, one has:
\[\pi_i[\kat\gamma] = \pi_i\gamma[0] = \rho_{L_i}(\pi_i\gamma)[0] = \kat\rho_{L_i}(\pi_i\gamma)\leq\Kat(L_i) \leq F.\]
Thus 
\[\kat\gamma \le \sum_i\pi_i[\kat\gamma] + F \le F.\]

Similarly, for $\delta\in\Delta$, one has:
\[\pi_i[\kat\delta] = \pi_i\delta[0]\le L_i\cap(\delta\pi_i[0] + \kat\pi_i) = \check\delta[\kat\check\pi_i] + \Kat(L_i) = \Kat(L_i) \leq F.\]
It follows as above that $\kat\delta \le F$.%, and $\kat_\Delta(A)$ is finite. Thus $\kat_{\Gamma,\Delta}(A)$ is finite as well.
\qedhere
\end{enumerate}
\end{proof}

\subsection{The induction}\label{s:induction}

We finally prove our main result, recalled for convenience.

\begin{theorema}
In a finite-dimensional theory, let $A$ be a definable, virtually connected, abelian group. Let $\Gamma$ and $\Delta$ be two invariant prerings of definable endogenies of $A$ such that $A$ is $(\Gamma,\Delta)$-minimal. Suppose $\Gamma$ and $\Delta$ commute sharply, both are essentially infinite, and one of them is essentially unbounded.
Then $\Kat(\Gamma,\Delta)$ is finite.
\end{theorema}
\begin{proof}
Suppose that $A$ is a counterexample of minimal dimension.

We replace $\Gamma$ by $\Gamma'=\CC(\Delta)\ge\Gamma$. This clearly preserves all the hypotheses, and $\Kat(\Gamma,\Delta)\le\Kat(\Gamma',\Delta)$. So we can assume $\Gamma=\CC(\Delta)$ and $\Delta = \CC(\Gamma)$.
If $\Gamma$ and $\Delta$ act with finite kernels, then $\Kat_{\Gamma, \Delta}(A)$ is finite by Proposition~\ref{t:local}. Otherwise, there is $\varphi\in\Gamma\cup\Delta$ with infinite kernel. By symmetry, we may assume $\varphi\in\Gamma$.

Any line $L \in \Lambda$ has dimension $\dim L\le \dim\varphi[A]<\dim A$ and is virtually connected by Proposition~\ref{p:lines}.\ref{i:Lexistence}. The prerings $\Gamma_L$ and $\Delta_L$ sharply centralize one another by Proposition~\ref{p:GammaLDeltaL}.\ref{i:GammaLDeltaLcommute}. Moreover $L$ is $(\Gamma_L, \Delta_L)$-minimal by Proposition~\ref{p:GammaLDeltaL}.\ref{i:lineminimal}, and $\Gamma_L$ and $\Delta_L$ have essentially the same size as $\Gamma$ and $\Delta$ by Proposition~\ref{p:GammaLDeltaL}.\ref{i:Gammalarge}.

By induction, $\Kat(\Gamma_L,\Delta_L)$ is finite for any line $L$. But then $\Kat(\Gamma,\Delta)$ is finite by Proposition~\ref{p:linereduction2}.\ref{i:radAfini}.
\end{proof}

\section{Linearisation (proof of Theorem~B)}\label{S:linearisation}

In this section we prove Theorem~B. First, we restate the expected definition of minimality for a bi-module under endomorphisms.

\begin{definition}
Let $A$ be a definable, connected, abelian group. Let $\Gamma$ and $\Delta$ be two commuting rings of endomorphisms.
% $(\Gamma, \Delta)$-invariant group.
The group $A$ is $(\Gamma, \Delta)$-minimal if it has no non-trivial, proper, connected, definable $(\Gamma, \Delta)$-invariant subgroup of infinite index.
\end{definition}

\begin{theoremb}
In a finite-dimensional theory, let $A$ be a definable, connected, abelian group. Let $\Gamma$ and $\Delta$ be two invariant rings of definable endomorphisms of $A$ such that $A$ is $(\Gamma, \Delta)$-minimal. Suppose that $\Gamma$ and $\Delta$ commute, both are infinite and one of them is unbounded. Then there is a definable field $\bK$ such that $A$ is definably a $\bK$-vector space of finite linear dimension, and both $\Gamma$ and $\Delta$ act $\bK$-linearly.
\end{theoremb}

Similar to the proof of Theorem A, the proof will proceed by (double) reduction to lines. However, the fact that we are working with endomorphisms rather than endogenies will allow for more precision.
\begin{remark}\label{r:centraliser}\begin{enumerate}\item If $A$ is merely virtually connected, then any endomorphism of $A$ restricts to an endomorphism of $A^\circ$. One {\em could} work with virtually connected $A$; in the course of the proof one would have to divide out by a finite subgroup $F$, and end up with $A/F$ a virtually connected definable vector space over a definable virtually connected field $\bK$. However, the connected component of a virtually connected field is an ideal, whence the whole field. So $\bK$ is connected, as is $A/F$.
\item\label{i:centraliser} If we replace $\Gamma$ by $C\Delta)\ge\Gamma$ and then $\Delta$ by $C(C(\Delta))\ge\Delta$, we obtain two commuting rings of endomorphism of $A$ containing the original $\Gamma$, $\Delta$, such that moreover each is the centraliser of the other.\end{enumerate}\end{remark}

\subsection{Local version (finite kernels)}

\begin{proposition}[Local version of Theorem B]\label{p:localbis}
Same setting as Theorem~B. Instead of assuming $(\Gamma, \Delta)$-minimality, suppose that every endomorphism in $\Gamma \cup \Delta$ has finite kernel.
Also suppose that $\Gamma = C(\Delta)$ and $\Delta = C(\Gamma)$.

Then there is a definable skew field $\bK$ such that $A$ is a $1$-dimensional $\bK$-vector space on which $\Gamma$ acts as $\bK$ and $\Delta$ as $\bK^\op$. Moreover, $\bK$ has finite degree over its centre $Z$, so $A$ is a finite-dimensional $Z$-vector space, and $Z\subseteq\Gamma\cap\Delta$.
\end{proposition}

\begin{remark} Non-commutativity of the skew field cannot be avoided, since a real closed field is dimensional and interprets a field of quaternions.\end{remark}

\begin{proof} By dimensionality and connectedness, all endomorphisms from $\Gamma\cup\Delta$ are surjective. So if one of them, say $\varphi$, has nontrivial kernel, then $B = \bigcup_{n<\omega} \ker (\varphi^n)$ is a countably infinite group. By symmetry, we may assume $\varphi\in\Delta$. Hence $B$ is $\Gamma$-invariant.
\begin{itemize}
\item
First suppose that $\Gamma$ is unbounded. Of course $\End(B)$ is bounded.
By unboundedness, there is $\gamma \in \Gamma\setminus\{0\}$ with $B \leq \ker \gamma$, a contradiction to kernel finiteness.
\item
Now suppose that $\Gamma$ is bounded. By assumption, $\Delta$ is not; so by the above, $\Gamma$ acts with trivial kernels. But $\ker\varphi$ is finite $\Gamma$-invariant and $\Gamma$ is infinite. So there is $\gamma \in \Gamma$ with $\ker\varphi\leq \ker \gamma$, a contradiction to kernel triviality.
\end{itemize}
So $\Delta$ acts with trivial kernels, as does $\Gamma$ by symmetry. It follows that $\dim X\le \dim A$ for all type-definable $X\subseteq\Gamma$ or $X\subseteq\Delta$, as any $\varphi\in\Gamma\cup\Delta$ is definable over $(a,\varphi(a))$ for any $a\in A\setminus\{0\}$. Thus $\Gamma$ and $\Delta$ are Ore rings, and their fraction skew fields $\bK$ and $\bL$ are type-definable by \cite{wa20}. Moreover $\bK, \bL$ still act by automorphisms on $A$. Now $\bK = \Frac \Gamma\leq C(\Delta) = \Gamma \leq \bK$, so $\Gamma = \bK$ and $\Delta = \bL$ are skew-field of automorphisms, and $A$ is a vector space over each.

Linear dimensions are finite since $\bK$ and $\bL$ are infinite. If the linear dimension $\dim_\bK A$ were $> 1$, then there would exist a non-zero, non-injective, $\bK$-linear endomorphism $f$ of $A$. But the dimension being finite, $f$ is definable: hence $f \in C_{\DefEnd(A)}(\bK) = \bL$, a contradiction. So linear dimensions are $1$. Finally $\bL = \bK^\op$ acting on $A \simeq \bK_+$.

By \cite{wa20}, a dimensional skew field has finite dimension over its centre. The rest follows.
\end{proof}

\subsection{Lines, and local endomorphisms}
As in the last section, we first define {\em lines}. We first use the same definition so work is not lost; there will be an equivalent, easier definition.

\begin{definition}
A \emph{line} is a $\Gamma$-image $\gamma[A]$ which is infinite and minimal with respect to inclusion,
% A $\Gamma$-image $\gamma[A]$ is a {\em line} if it is minimal infinite,
i.e.~for all $\gamma'\in\Gamma$, if $\gamma'[A]\le\gamma[A]$ is infinite, then $\gamma'[A]=\gamma[A]$.
We denote the set of lines by $\Lambda$.
\end{definition}

\begin{remark}
Let $k = \min \{\dim \gamma[A]: \gamma \in \Gamma\wedge \gamma[A]\gg 0\}$. Then $\Lambda = \{\gamma[A]: \gamma \in \Gamma \wedge \dim \gamma[A] = k\}$.

Indeed, images $\gamma[A]$ are now all connected since $A$ is connected and $\Gamma$ consists of definable endomorphisms. Moreover lines all have the same dimension by Proposition~\ref{p:lines}.\ref{i:finitekernel}.
\end{remark}

\begin{definition}
For $\varphi$ an endomorphism stabilising $L$ we let $\check\varphi = \rho_L(\varphi) = \varphi \cap L^2$ be its restriction-corestriction to $L$. It is an endomorphism of $L$.

Let $\Gamma_L = \{\check\gamma: \gamma \in \Gamma \wedge \im\gamma\le L\}\leq \End(L)$, and let $\Delta_L = \{\check\delta:\delta \in \Delta\} \leq \End(L)$. They are subrings.
\end{definition}

Note that lines are connected and $\Delta$-invariant, and $\Gamma_L$ commutes with $\Delta_L$.
Note that $\check\gamma\in\Gamma_L$ for any $\gamma\in\gamma_0\Gamma$, for $L = \gamma_0[A]$. By dimensionality, any $\gamma\in\Gamma$ with $\gamma[L] \gg 0$ has $L\cap\ker\gamma$ finite, and $\gamma[L]$ is again a line. In particular, any non-zero $\gamma\in\Gamma_L$ is surjective with finite kernel.

Let `large' mean either infinite or unbounded.

\begin{proposition}\label{p:linereductionL}
Let $L = \gamma_0[A]\in\Lambda$ be any line. Then the following hold.
\begin{enumerate}
\item\label{i:lineminimalL}
$L$ is $(\Gamma_L, \Delta_L)$-minimal.
\item\label{i:AsumoflinesL}
$A$ is a finite sum of $\Gamma$-images of $L$.
\item\label{i:lineinimageL}
If $\gamma[A] \gg 0$, then there is $\gamma' \in \Gamma$ such that $\gamma\gamma'[A]$ is a line contained in $\gamma[A]$.
\item\label{i:transitiveL}
If $L_2 = \gamma_2[A]$ and $L_1 \in \Lambda$, then there is $\gamma \in \Gamma$ such that $\gamma_2\gamma[L_1] = L_2$.
\item\label{i:setidempotentL}
There is $\gamma \in \Gamma$ such that $L = \gamma[A] = \gamma[L]$. Then $\check\gamma$ is surjective with finite kernel.
\item\label{i:GammalargeL} $\Gamma_L$ is as large as $\Gamma$; and $\Delta_L$ as $\Delta$.
\item\label{i:DeltalargeL} The restriction-corestriction map $\rho_L:\Delta\to\Delta_L$ is bijective.
\item\label{i:decompositionL} $A$ is a finite, direct sum of lines.
\end{enumerate}
\end{proposition}
\begin{proof} Items~\ref{i:lineminimalL}--\ref{i:DeltalargeL} follow from Proposition~\ref{p:lines}, since it also applies to rings of endomorphisms.

To show item \ref{i:decompositionL}, note that by Proposition \ref{p:linereduction2}.\ref{i:pi}, for any line $L$ there is a projection $\pi_L:A\to L$ in $C(\Delta) = \Gamma$ (note that there are no katakernels). Then $(1-\pi_L)[A] = H$ is a $\Delta$-invariant complement for $L$.

Now suppose that lines $L_0,\ldots, L_i$ are given with direct complements $H_i$ such that $L_j \le \bigcap_{k<j}H_k$ for all $j\le i$. Then one sees by induction that
\[A = \Big(\bigoplus_{j\le i}L_j\Big)\oplus\Big(\bigcap_{j\le i}H_j\Big).\]
The projection $\pi:A\to\bigcap_{j\le i}H_j$ with kernel $\bigoplus_{j\le i}L_j$ is $\Delta$-covariant, whence in $C(\Delta) = \Gamma$. If $\bigoplus_{j\le i}L_j<A$, then $\im\pi$ is infinite, and by~\ref{i:lineinimageL}.\ there is $\gamma\in\Gamma$ such that $\pi\gamma[A]$ is a line $L_{i + 1} \le \bigcap_{j\le i}H_j$, and we can iterate. The process must terminate by finite-dimensionality, and when it does one has $A = \bigoplus_i L_i$.
\end{proof}

\subsection{Local endomorphisms, continued}
\begin{proposition}\label{p:linereduction2L}
Suppose that $A$ is $(\Gamma, \Delta)$-minimal, and that for every line $L$ all surjective endomorphisms in $\Gamma_L\cup\Delta_L$ have trivial kernel.\begin{enumerate}
\item\label{i:trivialkernel} If $\gamma\in\Gamma$ and $\gamma[L]$ is a line $L'$, then $L\cap \ker\gamma$ is trivial.
\item\label{i:invertibleL} If $L$ and $L'$ are lines and $\varphi:L\cong L'$ is a $\Delta$-covariant isomorphism, then there is invertible $\gamma\in\Gamma$ inducing $\varphi$ on $L$.
\item All surjective endomorphisms in $\Gamma\cup\Delta$ have trivial kernel.
\item\label{i:lifting} $\Gamma_L = C(\Delta_L)$ and $\Delta_L = C(\Gamma_L)$.
Moreover, for any $\varphi\in\Gamma_L\cap\Delta_L$ there is $\hat\varphi\in\Gamma\cap\Delta$ with $\rho_L(\hat\varphi) = \varphi$.
\end{enumerate}\end{proposition}
\begin{proof}\begin{enumerate}
\item Take any $\gamma'$ with $\gamma'[L'] = L$. Then $\rho_L(\gamma'\gamma)\in\Gamma_L$ is surjective, and has trivial kernel by assumption. Hence $\ker\gamma\cap L$ is trivial.
\item Consider a direct $\Delta$-invariant complement $H$ for $L$, together with the projection $\pi_L:A\to L$ along $H$. So $A = L\oplus H$. Note that
$\pi_L[L']$ is $0$ if $L'\le H$, and $L$ otherwise by connectivity, in which case the kernel $L'\cap H$ must be trivial by \ref{i:trivialkernel}.

If $L'\cap H = \{0\}$, we can put $\gamma = \varphi\pi_L + (1-\pi_L)$. Then $\gamma$ equals $\varphi$ on $L$, and $1$ on $H$. Moreover $\gamma$ commutes with $\Delta$, so $\gamma\in C(\Delta) = \Gamma$. We have $\gamma^{-1} = \varphi^{-1}\pi_{L'} + (1-\pi_{L'})$, so $\gamma$ is invertible.

If $L'\le H$, we can suppose $L = L_0$ and $L' = L_1$ in a direct sum decomposition $A = \bigoplus_{i < n} L_i$ given by Proposition \ref{p:linereductionL}.\ref{i:decompositionL}. We can then take
\[\gamma = \sigma\pi_0 + \sigma^{-1}\pi_1 + (1-\pi_0-\pi_1).\]
This is $\Delta$-covariant, and $\gamma^{-1} = \gamma$.
\item Recall that there are finitely many lines $L_i$ such that $A = \bigoplus_i L_i$. 

Consider a surjective $\delta\in\Delta$. Then $\delta$ has finite kernel by dimensionality, so $\rho_{L_i}(\delta)\in\Delta_{L_i}$ is surjective, and has trivial kernel by assumption. It follows that $\delta = \sum_i\delta\pi_i$ has trivial kernel, since $\delta[L_i]\le L_i$ for all $i$.

Next, consider surjective $\gamma\in\Gamma$. Then $\gamma$ has finite kernel, so $\gamma[L_i] = L'_i$ is again a line for any $i$, and $L_i\cap\ker\gamma$ must be trivial by \ref{i:trivialkernel}.

But $\sum_iL'_i = \gamma[\sum_iL_i] = \gamma[A] = A$. By dimensionality, $\sum_iL'_i$ must be almost direct. So suppose $i$ is minimal such that $L'_i\cap\sum_{j<i}L'_j$ is finite non-trivial. Then $(1-\sum_{j<i}\pi_{L_j})[L']$ must be a line $L''$, and $L'\cap\sum_{j<i}L'_j = L'\cap\ker(1-\sum_{j<i}\pi_{L_j})$ is trivial by \ref{i:trivialkernel}, a contradiction. It follows that the sum $\sum_iL'_i$ is direct, and $\gamma$ has trivial kernel.
\item If $\varphi:L\to L$ is $\Delta_L$-covariant, then define $\gamma:A\to A$ by $\gamma = \varphi\pi_L$. Then $\gamma\in C(\Delta) = \Gamma$, and $\check\gamma = \varphi$, so $\varphi \in \Gamma_L$.

Now let $\varphi:L\to L$ be $\Gamma_L$-covariant. We need to find $\delta\in C(\Gamma) = \Delta$ with $\check\delta = \varphi$.

For any line $L'$ choose $\gamma\in\Gamma$ with $\gamma[L] = L'$.
Then $\gamma$ induces an isomorphism $L'\cong L$ by \ref{i:trivialkernel}. By \ref{i:invertibleL}.\ we may assume that $\gamma\in\Gamma$ is invertible. Now $\gamma\varphi\gamma^{-1}$ leaves $L'$ invariant; let $\varphi_{L'}$ be the induced endomorphism of $L'$. If $\gamma'\in\Gamma$ is invertible and also induces an isomorphism $L\cong L'$, then $\rho_L(\gamma^{-1}\gamma')\in\Gamma_L$ commutes with $\varphi$. Thus
\[\gamma'\varphi\gamma'^{-1} = 
\gamma(\gamma^{-1}\gamma')\varphi\gamma'^{-1} = \gamma\varphi(\gamma^{-1}\gamma')\gamma'^{-1} = \gamma\varphi\gamma^{-1},\]
so $\varphi_{L'}$ does not depend on the choice of $\gamma$.

Consider non-zero $\check\gamma''\in\Gamma_{L'}$. By \ref{i:invertibleL} we may assume that it is induced by an invertible $\gamma''\in\Gamma$ stabilising $L'$. Then $\gamma''\gamma$ is another invertible element of $\Gamma$ inducing an isomorphism $L\cong L'$, so 
\[\gamma''\varphi_{L'}\gamma''^{-1} = (\gamma''\gamma)\varphi(\gamma''\gamma)^{-1} = \varphi_{L'},\]
and $\varphi_{L'}$ commutes with $\check\gamma''$.

In fact, if $\gamma_1$ induces an isomorphism $L'\to L''$, then $\varphi_{L''}\gamma_1 = \gamma_1\varphi_{L'}$ on $L'$. To see this, consider invertible $\gamma_2,\gamma_3\in\Gamma$ inducing isomorphisms $\gamma_2:L\cong L'$ and $\gamma_3:L\cong L''$. Then $\gamma_1\gamma_2:L\cong L''$ is another induced isomorphism, and on $L'$ we have
\[\varphi_{L''}\gamma_1 = (\gamma_1\gamma_2)\varphi(\gamma_1\gamma_2)^{-1}\gamma_1 = \gamma_1(\gamma_2\varphi\gamma_2^{-1}) = 
\gamma_1\varphi_{L'}.\]

Finally, let $A = \bigoplus_iL_i$ be a direct decomposition of $A$ into lines, and define
\[\hat\varphi(\sum_i\ell_i) = \sum_i\varphi_{L_i}(\ell_i).\]
This is definable, well-defined and extends $\varphi$. Now if $\gamma\in\Gamma$ and $\gamma_{i,j} = \pi_{L_i}\gamma\pi_{L_j}$, then $\gamma_{i,j}$
is either $0$ or an isomorphism from $L_j$ to $L_i$. Hence
\[\hat\varphi\gamma(\sum_j\ell_j) = \sum_{i,j}\varphi_{L_i}(\pi_{L_i}\gamma\pi_j(\ell_j))
 = \sum_{i,j}\pi_{L_i}\gamma\pi_j\varphi_{L_j}\ell_j
 = \sum_j\gamma\varphi_{L_j}\ell_j = \gamma\hat\varphi(\sum_j\ell_j).\]
It follows that $\hat\varphi\in C(\Gamma) = \Delta$, and $\varphi = \rho_L(\hat\varphi)$, whence $\varphi\in\Delta_L$.

Note that if $\varphi$ also commutes with $\Delta_L$, then $\hat\varphi$ commutes with $\Delta$. Thus $\varphi\in\Gamma\cap\Delta$. 
\qedhere
\end{enumerate}
\end{proof}

\subsection{The induction}
We can now prove Theorem B. We recall its statement.
\begin{theoremb}
In a finite-dimensional theory, let $A$ be a definable, connected, abelian group. Let $\Gamma$ and $\Delta$ be two invariant rings of definable endomorphisms of $A$ such that $A$ is $(\Gamma, \Delta)$-minimal. Suppose that $\Gamma$ and $\Delta$ commute, both are infinite and one of them is unbounded. Then there is a definable field $\bK$ such that $A$ is definably a $\bK$-vector space of finite linear dimension, and both $\Gamma$ and $\Delta$ act $\bK$-linearly.
\end{theoremb}

\begin{proof} By Remark~\ref{r:centraliser}.\ref{i:centraliser}.\ we may assume that $\Gamma = C(\Delta)$ and $\Delta = C(\Gamma)$.
We shall use induction on the dimension $\dim A$.
If $A$ has no proper $\Gamma$- or $\Delta$-lines, all endomorphisms in $\Gamma\cup\Delta$ have finite kernel by dimensionality, and we finish by Proposition \ref{p:localbis}.

Otherwise, we can apply the inductive hypothesis to any line $L$, with $C(\Delta_L)\ge\Gamma_L$ and $C(C(\Delta_L))\ge\Delta_L$ acting on it. So there is a field $\bK_L\le C(\Delta_L)\cap C(C(\Delta_L))$ such that $L$ is a finite-dimensional $\bK_L$-vector space, on which both $C(\Delta_L)$ and $C(C(\Delta_L))$ act $\bK_L$-linearly. In particular all surjective endomorphisms in $C(\Delta_L)\cup C(C(\Delta_L))$ have trivial kernel. By Proposition \ref{p:linereduction2L} we have $C(\Delta_L) = \Gamma_L$ and $C(C(\Delta_L)) = \Delta_L$, so $\bK_L \le \Gamma_L\cap\Delta_L$.

By Proposition \ref{p:linereduction2L}.\ref{i:lifting}., for every $\varphi\in\Gamma_L\cap\Delta_L$ there is $\hat\varphi\in\Gamma\cap\Delta$ restricting to $\varphi$ on $L$. Consider $\bK = \langle\hat\varphi:\varphi\in \bK_L\rangle\subseteq\Gamma\cap\Delta$, a commutative subring. But the restriction-corestriction map $\rho_L:\Delta\to\Delta_L$ is a bijection. It follows that $\bK$ is a definable field, $A$ is a finite-dimensional $\bK$-vector space, and both $\Gamma$ and $\Delta$ act $\bK$-linearly.\end{proof}

\end{document}